\documentclass[12pt, a4paper]{amsart}
\usepackage[utf8]{inputenc}
\usepackage[english]{babel}
\usepackage{amsthm,amsfonts,amsmath, amssymb, mathtools}
\usepackage{indentfirst}
\usepackage{graphicx}
\usepackage{comment}
\usepackage{csquotes}
\usepackage{enumerate}
\usepackage{tikz}
\usepackage[all,cmtip]{xy}
\usepackage{multicol}

\usepackage[normalem]{ulem} 
\usepackage{url}
\usepackage{hyperref}
\hypersetup{
    colorlinks=false,
    linkcolor=blue, 
    filecolor=magenta,      
    urlcolor=cyan,
    citecolor=orange,
    pdftitle={Revisiting Arnold's topological proof of the Morse index theorem},
    pdfpagemode=FullScreen,
    }
\usepackage{amsrefs}

\newcommand{\CC}{\mathbb{C}}

\newcommand{\RR}{\mathbb{R}}     

\newcommand{\ZZ}{\mathbb{Z}}

\DeclareMathOperator{\Ind}{Ind}
\DeclareMathOperator{\ind}{ind}
\DeclareMathOperator{\nul}{nul}
\DeclareMathOperator{\Sym}{Sym}

\DeclareMathOperator{\Hom}{Hom}
\DeclareMathOperator{\colsp}{colsp}
\DeclareMathOperator{\Det}{Det}

\newcommand{\RP}{\mathbb{R}\mathrm{P}}

\newtheorem{teo}{Theorem}[section]
\newtheorem{lema}[teo]{Lemma}

\newtheorem{prop}[teo]{Proposition}
\theoremstyle{definition}

\title[Arnold's proof of the Morse index theorem]{Revisiting Arnold's topological\\
proof of the Morse index theorem}
\author[E. V. Sodré]{Eduardo V. Sodré}
\address{Departamento de Matemática, Universidade de São Paulo, Brazil}
\curraddr{Department of Mathematics, Brown University, USA}

\email{eduardo\_sodre@brown.edu}
\date{July 2\textsuperscript{nd} 2023}

\begin{document}

\maketitle
\begin{abstract}

We give an exposition of the Morse Index Theorem in the Riemannian case in terms of the Maslov index, following and expanding upon Arnold's seminal paper. We emphasize the symplectic arguments in the proof and aim to be as self-contained as possible.

\end{abstract}

\section{Introduction}

Carl Gustav Jacobi in 1842 \cite{jacobi} seems to have been the first to investigate whether the principle of least action in the calculus of variations always yields minima as opposed to other kinds of stationary points, and most of his work focused on geodesics in two-dimensional surfaces. Marston Morse gave the first clear general statements in the 1920s and 1930s, leading to what is known today as Morse theory \cite{morse2}.
In particular, his celebrated Index Theorem roughly states that the number of essentially shorter routes near a given geodesic can be computed as the number of conjugate points along the geodesic (counting multiplicity).
Perhaps as a sign of the depth of this statement, one can find in the literature countless variations, extensions and generalizations of Morse's index theorem, together with methods of proofs of different flavors. As Bott put it in \cite{bott}, ``in all properly posed variational problems there is some kind of index theorem''.

Beautiful as finite-dimensional Morse theory can be, the real goal of Morse was the infinite-dimensional calculus of variations setting of Morse theory. He considers a functional
\[
J(\sigma) = \int_a^b F(\sigma,\dot\sigma)\,dt \]
in some space of paths \( \Omega \), subject to a certain nondegeneracy condition and admissible boundary conditions. Here the ``tangent space'' to an extremal \( \sigma \) of \( J \) is the set of vector fields along it, and the ``Hessian'' of \( J \) at \( \sigma \) is given by the second variation of \( J \). By choosing a frame along \( \sigma \) a vector field along \( \sigma \) is identified with an \( \mathbb R^n \)-valued function of the parameter \(t\) along \(\sigma\) and, upon integration by parts,
that Hessian takes the form
\[ 
\int_a^b \langle Lx,x\rangle\,dt, 
\]
where \( x(t) \) represents a vector field along \( \sigma \) and \( \langle,\rangle \) denotes the pointwise inner product, for a self-adjoint second order linear differential operator \( L \). The Sturm-Liouville eigenvalue problem
\[ 
Lx = \lambda x 
\]
subject to boundary conditions turns out to be well-posed and thus has a finite-dimensional solution space for \( \lambda \leq 0 \). Morse proceeds to define the index and nullity for \( \sigma \) respectively as the dimension of the space of solutions of \( Lx = 0 \) and the dimension of the space of solutions of \( Lx = \lambda x \) with \( \lambda < 0 \). 

In Riemannian geometry, for \( J \) one takes the energy functional on a complete Riemannian manifold \( M \), given by 
\[ 
E(\gamma) = \int_a^b ||\dot\gamma||^2\,dt 
\]
and defined on the space \(\Omega\) of piecewise smooth curves parametrized by \( t \in [a,b] \) proportional to arc-length, and the boundary conditions are the fixed endpoint conditions
\[
\gamma(a) = p, \quad \gamma(b) = q,
\]
for fixed \(p\), \(q \in M\). The extremals are exactly the geodesics, and here the eigenvalue problem presents itself as
\[
-Y'' + R(\dot\gamma,X) \dot\gamma = \lambda X, \qquad Y(a) = Y(b) = 0,
\]
where the prime denotes covariant differentiation of the vector field \( Y \) along \( \gamma \) and \( R \) denotes the curvature tensor. The index of \( \gamma \) manifests itself as the obstruction to the geodesic being a local minimum of the energy, as subspaces on which the Hessian is negative definite are the directions on which we can perturb \( \gamma \) and obtain shorter paths. The solutions in case \(\lambda = 0 \) are called \emph{Jacobi fields}, and points \(p\) and \(q\) are called \emph{conjugate along \(\gamma\)} in case there is a nonzero Jacobi field vanishing at \(a\) and \(b\); in this case the \emph{multiplicity} of such a conjugate pair is the dimension of the space of such Jacobi fields. He then succeeds in proving the Morse inequalities for any nondegenerate \( J \), and arrives at his beautiful Index Theorem: \emph{the index of an extremal \( \gamma \) in the fixed endpoints case   equals the number of conjugate points of one endpoint in the interior of \( \gamma \), counting multiplicity.}

Morse himself applied the Index Theorem to obtain deep results about existence of geodesics in the $2$-sphere with an arbitrary metric, and Bott \cite{bott2} was led by a similar analysis to his celebrated Periodicity Theorem. Morse's original proof of the Index Theorem has been expounded by Ambrose \cite{ambrose} and made concise by Osborn \cite{osborn}, and has been generalized to PDEs by Smale \cite{smale} and minimal submanifolds by Simons \cite{simons}. Uhlenbeck \cite{uhlenbeck} gave a proof based on Hilbert spaces and applied it to minimal submanifolds as well. It has also been presented accesibly in book form by Milnor \cite{milnor} and further divulged by do Carmo \cite{docarmo}. On a different vein, the Morse theory in Hilbert spaces has beend developed by Palais \cite{palais} (see also the textbook by Klingenberg \cite{klingenberg}). The Index Theorem represent a natural extension of the classical Sturm-Liouville theory of differential equations and, as such, it has been considered by Edwards \cite{edwards} and Zhu \cite{zhu} by adapting it to higher order systems. It was also applied in pseudo-Riemannian geometry by Helfer \cite{helfer} and by Giannoni, Masiello, Piccione and Tausk \cites{giannoni, piccione}, such as in the case of conjugate points along spacelike geodesics. For a $K$-theoretic approach to Morse's Index Theorem, see \cite{waterstraat}. Similar ideas are used in a generalization of Morse theory called Floer theory \cites{floer, robbin}. 

We were particularly attracted to Arnold's original paper \cite{arnold}, displaying an ingenuity and simplicty so characteristic of him, and this text is our attempt to present his arguments from our point of view. From a modern perspective, and considering the self-adjointness of the Jacobi operator,  we want to make evident the almost inevitablity of the appearence of symplectic methods, revolving around the Maslov-Arnold index. Rephrasing the result in a topological way, in terms of intersections of Lagragian subspaces, opens up new venues and vistas. This is an approach also taken by Duistermaat \cite{duistermaat} and Lytchak \cite{lytchak}, and by Piccione and Tausk \cite{piccione2}.

Our aim has been to follow a most ``natural'' path (not always the shortest one), based on elementary arguments and simplified constructions, and to be as transparent as possible. For this reason we also restrict the discussion to the most basic case, that is, Riemannian geodesics with fixed endpoints.

We now sketch the main issues involved in our exposition of Arnold's ideas. 
In the Riemannian case, consider a geodesic \( \gamma \) defined on an interval \( [a,b] \), and denote by \( H_t \) the Hessian of the energy functional defined on the space of vector fields along \(\gamma|_{[a,t]}\) vanishing at \(a\) and \(t\). We like to think of the Index Theorem as the following
chain of equalities:
\[
\mathrm{ind}(H_b) = \sum_{\lambda<0}\mathrm{nul}(H_b-\lambda I)
  = \sum_{\lambda\in(\lambda_0,0)}\mathrm{nul}(H_b-\lambda I)
  = \sum_{t\in(a,b)}\mathrm{nul}(H_t),
\]
where \( \lambda_0 \) is some negative number. The first equality would be clear in finite dimensions, but requires some discussion in infinite dimensions. The second equality is due to the fact that the corresponding Sturm-Liouville problem has eigenvalues bounded below.
Indeed, it follows from standard Sturm-Liouville theory that there are finitely many negative eigenvalues, but we circumvent this extra background in our topological approach.

The last equality is at the core of our discussion, and is obtained from interpreting the relevant nullities as intersection numbers of a certain \(1\)-cycle with the canonical Maslov cycle in the Lagrangian Grassmannian. This \(1\)-cycle is a homologically trivial curve of Lagrangian subspaces constructed from the Jacobi equation, hence the total intersection number must vanish, from which we derive the Index Theorem.

\section{Riemannian Geometry}

With the notation used in the introduction, let \( \Gamma \) the set of smooth vector fields along a geodesic \( \gamma: [a,b] \to M \) and \( \Gamma_0 \subset \Gamma \) those vector fields which vanish at the endpoints. The index form \( H_b: \Gamma_0 \times \Gamma_0 \to \RR \), arising from the second variation of the energy, is given by \begin{align} \label{indexform1}
H_b(X,Y) & =  \int_a^b \langle X', Y' \rangle + \langle R(\dot{\gamma}, X)\dot{\gamma}, Y \rangle ds \\ \label{indexform2}
& = \int_a^b \langle - X'' + R(\dot{\gamma}, X)\dot{\gamma}, Y\rangle ds.
\end{align}
It is bilinear and symmetric, and we naturally consider those \( X \in \Gamma \) that satisfy
\[
-X'' + R(\dot{\gamma}, X)\dot{\gamma} = 0,
\]
being the Jacobi fields. Note that \( R(t) \coloneqq R(\dot{\gamma}(t), \cdot)\dot{\gamma}(t) \) is a self-adjoint operator on \( T_{\gamma(t)}M \) due to the symmetries of the curvature tensor. By choosing a parallel orthonormal frame \( (E_1, \ldots, E_n) \) along \( \gamma \), the Jacobi fields \( X(t) = x^i(t)E_i(t) \) correspond to solutions of a homogenous second order linear system of ODEs. They are smooth and form a vector space \( \mathcal{J} \subset \Gamma \) of dimension \( 2n \), being uniquely defined by any prescribed pair of values \( (X(t), X'(t)) \) for \( t \in [a,b] \), in particular the initial conditions \( (X(a), X'(a)) \). It is also easily seen from (\ref{indexform2}) that the kernel of \( H_b \) as a symmetric bilinear form is exactly \( \mathcal{J} \cap \Gamma_0 \), that is, the set of Jacobi fields which vanish at the endpoints \( a \) and \( b \).

We say that \( t \in (a,b] \) is a \textit{conjugate value} to \( a \) along \( \gamma \), and that \( \gamma(t) \) is its respective \textit{conjugate point}, if there exists a non-zero Jacobi field \( X \) along \( \gamma|_{[a,t]} \) such that \( X(a) = X(t) = 0 \). This field can naturally be extended to a Jacobi field defined on the whole interval \( [a,b] \). Recall that the index of a symmetric bilinear form is the maximal dimension of a subspace on which it is negative definite. This dimension can, in principle, be infinite. We also know that the kernel of the index forms \( H_t \) for \( t \in [a,b] \) are the Jacobi fields that vanish at \( a \) and \( t \). The Morse Index Theorem, as stated previously, asserts that the index of \( H_b \) is equal to the number of conjugate values to \( a \) in \( (a,b) \) along \( \gamma \) counted with their multiplicity:

\begin{teo}[The Morse Index Theorem] \label{morseindex}
\[ \mathrm{ind}(H_b) = \sum_{\lambda<0}\mathrm{nul}(H_b-\lambda I)
  = \sum_{\lambda\in(\lambda_0,0)}\mathrm{nul}(H_b-\lambda I)
  = \sum_{t\in(a,b)}\mathrm{nul}(H_t). \]
\end{teo}

The first identity affirms that the index of \( H_b \) corresponds to the number of negative eigenvalues of the Sturm-Liouville problem
\begin{equation} \label{sturmliouville}
\begin{cases}
L_{\lambda}[X] = - X'' + (R-\lambda I)X = 0, \\
X(a) = X(b) = 0
\end{cases}
\end{equation}
counted with multiplicity, and the third identity represents the equivalence with the conjugate values with multiplicity. At first, we don't necessarily know whether the index, the number of negative eigenvalues, and the number of conjugate values are finite, but we can promptly prove the second identity:

\begin{teo}
The eigenvalues of the Sturm-Liouville problem
\[
\begin{cases}
L_{\lambda}[X] = - X'' + (R-\lambda I)X = 0, \\
X(a) = X(t) = 0
\end{cases}
\]
are bounded below by some \( \lambda_0 \) that does not depend on \( t \in (a,b] \).
\end{teo}

\begin{proof}
If \( X \) is a solution, then \( 0 = \langle L_{\lambda}[X], X \rangle \). We take the integral over \( [a,t] \), integrate by parts and use that \( X(a) = X(t) = 0 \) to obtain
\begin{gather*}
0 = \int_a^t \|X'\|^2 + \langle (R-\lambda I)X, X \rangle ds.
\end{gather*}

For \( -\lambda \) sufficiently large, the self-adjoint operator \( R(t) - \lambda I \) will have only positive eigenvalues for all \( t \in [a,b] \). If \( \mu \) is the infimum of the lowest eigenvalue of \( R(t) - \lambda I \) for \( t \in [a,b] \), then
\[
0 = \int_a^t \|X'\|^2 + \langle (R-\lambda I)X, X \rangle ds \geq \int_a^t \|X'\|^2 + \mu\langle X, X \rangle ds \geq 0,
\]
showing that \( X' \equiv 0 \) and, in turn, that \( X \equiv 0 \).
\end{proof}

We call \( X \in \Gamma \) a \textit{\(\lambda \)-Jacobi field} if it satisifies the equation
\begin{equation}
-X'' + RX = \lambda X, \label{thetajacobi}
\end{equation}
for a given real parameter \( \lambda \). Analogously to Jacobi fields, they form a real vector space \( \mathcal{J}_{\lambda} \subset \Gamma \) of dimension \( 2n \), given uniquely by the initial conditions \( (X(a), X'(a)) \) and \( \mathcal{J}_0 = \mathcal{J} \). It is readily verifiable that if \( X, Y \in \mathcal{J}_{\lambda} \), the expression
\begin{equation} \label{symplecticform}
\omega(X,Y) = -\langle X(t), Y'(t) \rangle + \langle X'(t), Y(t) \rangle
\end{equation}
does not depend on \( t \in [a,b] \) and defines a bilinear non-degenerate antisymmetric form on \( \mathcal{J}_{\lambda} \), that is, a symplectic form. Given \( t \in [a,b] \), the set
\[
\mathcal{J}_{\lambda}^t \coloneqq \{ X \in \mathcal{J}_{\lambda} \mid X(t) = 0 \}
\]
is a Lagrangian subspace of \( \mathcal{J}_{\lambda} \). So, in a certain sense, to study the \( \lambda \)-Jacobi fields that satisfy \( X(a) = X(t) = 0 \) is to study the intersections of specific Lagragian subspaces of a real symplectic vector space.

With a parallel orthonormal frame \( (E_1, \ldots, E_n) \) of smooth vector fields along \( \gamma \), the vector fields \( X \in \Gamma \) along the geodesic are represented by
\[
X(t) = x^i(t)E_i(t) \longmapsto (x^1(t), \ldots, x^n(t)),
\]
and to each vector field \( X \in \Gamma \) we associate the curve \( Y: [a,b] \to \RR^{2n} \)
\[
Y(t) = (x^1(t), \ldots, x^n(t), (x^1)'(t), \ldots, (x^n)'(t)).
\]

This is convenient because the \( \lambda \)-Jacobi equation is second order, and it becomes equivalent to the system of ODEs
\begin{equation} \label{secondorderode}
Y'(t) = A(t,\lambda)Y(t), \quad A(t,\lambda) = \begin{bmatrix}
0 & I \\
R(t) - \lambda I & 0
\end{bmatrix}
\end{equation}
where \( Y: [a,b] \to \RR^{2n} \) is a curve and \( R(t) \) is a curve of symmetric bilinear forms on \( \RR^n \). This implies that \( Y \) is of the form \( Y(t) = \begin{bmatrix}
X(t) & X'(t) \end{bmatrix}^{\mathsf{T}} \), and \( X(t) = (x^1(t), \ldots, x^n(t)) \) produces a \( \lambda \)-Jacobi field along \( \gamma \).
Also, since \( R(t) - \lambda I \) is self-adjoint, \( A \) is in the Lie algebra \( \mathfrak{sp}(2n, \RR) \) of linear symplectic maps of \( \RR^{2n} \), so that the flow preserves the canonical symplectic form on \( \RR^{2n} \).

Let \( \sigma = \{ 0 \} \times \RR^n \) be fixed as a Lagrangian subspace, corresponding the initial condition of the \( \lambda \)-Jacobi field being \( 0 \) at \( t = a \). If we consider \( \sigma_{\lambda}(t) \) to be the flow of \( \sigma \) at time \( t \) with respect to the system of ODEs above, that is,
\begin{equation}
\sigma_{\lambda}(t) = \{ Y(t) \in \RR^{2n} \mid Y(a) \in \sigma, \ Y'(t) = A(t,\lambda)Y(t) \}, \label{sigmalambda}
\end{equation}
which is also a Lagrangian subspace of \( \RR^{2n} \), then the intersection \( \sigma_{\lambda}(t) \cap \sigma \) corresponds exactly to the \( \lambda \)-Jacobi fields such that \( X(a) = X(t) = 0 \). In particular,
\[
\nul(H_t - \lambda I) = \dim(\sigma_{\lambda}(t) \cap \sigma),
\]
and the last equality in the Morse Index Theorem \ref{morseindex} is equivalent to
\begin{equation} \label{homologicalrectangle}
\sum_{t \in (a,b)} \dim(\sigma_0(t) \cap \sigma) - \sum_{\lambda \in (\lambda_0, 0)} \dim(\sigma_\lambda(b) \cap \sigma) = 0.
\end{equation}

By rephrashing the statement of the Morse Index Theorem in terms of intersections of Lagrangian subspaces with a given Lagrangian, we can use known topological methods to prove the equality. More specifically, we will view the above equality as the intersection number of a curve with a given subset of the moduli space of Lagrangians on \( \RR^{2n} \).

\section{The Lagrangian Grassmannian}

Let \( \RR^{2n} \cong \CC^{n} \) be equipped with its usual complex strucutre \( J \), inner product \( \langle \cdot, \cdot \rangle \) and symplectic form \( \omega \) as a \( 2n \)-dimensional real vector space, with coordinates
\[
(q,p) = (q^1, \ldots, q^n, p^1, \ldots, p^n) = q + ip.
\]

The Lagrangian Grassmannian \( \Lambda = \Lambda(n) \), that is, the set of all Lagrangian subspaces of \( \RR^{2n} \), is an embedded compact submanifold of the Grassmannian \( G_n(\RR^{2n}) \) of \( n \)-dimensional subspaces of \( \RR^{2n} \). For example, since every line passing through the origin in \( \RR^2 \cong \CC \) is Lagrangian, we have \( \Lambda(1) \cong \RP^2 \).

To identity a set of charts for \( \Lambda(n) \), let \( \sigma = \{0\}\times \RR^n \cong i\RR^n \) and consider the chart for \( G_n(\RR^{2n}) \) given by
\[
\begin{array}{cccc}
    \phi: & \mathrm{M}(n,\RR) & \longrightarrow & G^0_m(\sigma)  \\
     & S & \longmapsto & \lambda_S \coloneqq \{(q, Sq)\}
\end{array}
\]
which takes an \( n \times n \) real matrix \( S \) to its graph, an \( n \)-dimensional subspace of \( \RR^{2n} \) transversal to \( \sigma \). Note that transversality is an open condition in \( G_n(\RR^{2n}) \). We check that \( \lambda_S \) is Lagrangian if and only if \( S \) is symmetric, as
\begin{gather*}
\omega((q,Sq), (r,Sr)) = -\langle q, Sr \rangle + \langle Sq, r \rangle = \langle q, (S^{\mathsf{T}} - S)r \rangle = 0,
\end{gather*}
for all \( q, r \in \RR^n \), so \( S = S^{\mathsf{T}} \). This provides a chart for the set \( \Lambda^0(\sigma) \) of Lagrangians \( \lambda \) transversal to \( \sigma \), that is, such that \( \dim(\lambda \cap \sigma) = 0 \):
\begin{equation} \label{chartlagrangian}
\begin{array}{cccc}
    \varphi: & \Sym(n,\RR) & \longrightarrow & \Lambda^0(\sigma)  \\
     & S & \longmapsto & \lambda_S \coloneqq \{(q, Sq)\}.
\end{array}
\end{equation}

If \( K \subseteq \{ 1, \ldots, n \} \) is a set of indices, we also construct the unitary transformations \( J_{K} : \RR^{2n} \to \RR^{2n} \) given by
\[
J_K(q^i, p^i) = \begin{cases}
(-p^i, q^i), & \mbox{ if } i \in K; \\
(q^i, p^i), & \mbox{ if } i \notin K,
\end{cases}
\]
corresponding to multiplication by \( i \) on the \( K \) coordinates of \( \CC^n \), and the Lagrangian subspaces \( \sigma_K = J_K \sigma \), given by
\[
\sigma_K = \{ (q, p) \mid p^i = 0, \forall i \in K, \ q^j = 0, \forall j \notin K \}.
\]

Since \( J_K\sigma = \sigma_K \), we have that \( J_K \Lambda^0(\sigma) = \Lambda^0(\sigma_K) \), the set of Lagrangians transversal to \( \sigma_K \), and we construct the maps
\begin{equation} \label{chartsK}
\begin{array}{cccc}
    \varphi_K = J_K \varphi: & \Sym(n,\RR) & \longrightarrow & \Lambda^0(\sigma_K) \\
     & S & \longmapsto & J_K \lambda_S.
\end{array}
\end{equation}

\begin{lema} \label{chart}
If \( \lambda \in \Lambda \) is such that \( \dim(\lambda \cap \sigma) = k \), there exists a set of indices \( K \subseteq \{ 1, \ldots, n \} \) such that \( |K| = k \) and \( \lambda \) is transversal to \( \sigma_K \), that is, \( \lambda \in \Lambda^0(\sigma_K) \).
\end{lema}

\begin{proof}
If \( \lambda_0 = \lambda \cap \sigma \), We show first that \( \lambda_0 \cap \sigma_K = \{ 0 \} \). If \( \{ v_1, \ldots, v_k \} \) is a basis of \( \lambda_0 \), we complete it to a basis of \( \sigma \) with canonical vectors \( e_{i_1}, \ldots, e_{i_{n-k}} \). Being \( I = \{ i_1, \ldots, i_{n-k} \} \) this set of indices, we can choose \( K = I \), with \( \sigma_K \) satisfying the transversality condition with \( \lambda \) in \( \sigma \).

Now let \( \tau = \sigma_K \cap \sigma \). As \( \tau \cap \lambda_0 = \{0\} \), we have the direct sum \( \lambda_0 \oplus \tau = \sigma \). Considering the symplectic form \( \omega \), we also have that \( \omega(\lambda, \lambda_0) = 0 \) and \( \omega(\sigma_K, \tau) = 0 \), since \( \lambda_0 \subseteq \lambda \) and \( \tau \subseteq \sigma_K \) and \( \lambda, \sigma_K \) are Lagrangian subspaces. Then \( \omega(\lambda \cap \sigma_K, \sigma) = 0 \), which implies that \( \lambda \cap \sigma_K \subseteq \sigma \); but since they are transversal in \( \sigma \), it must be that \( \sigma_K \cap \lambda = \{0\} \).
\end{proof}

\begin{teo}
The set \( \Lambda^k(\sigma) = \{ \lambda \in \Lambda \mid \dim(\lambda \cap \sigma) = k \} \) is covered by the \( \binom{n}{k} \) charts \( \varphi_K \), and on each such chart, the coordinates \( S = \varphi_K^{-1}\lambda \) for \( \Lambda^k(\sigma) \) are given by \( S_{\mu \nu} = 0 \), \( \forall \mu, \nu \in K \).
\end{teo}

\begin{proof}
Without loss of generality, we may assume that \( K = \{ 1, \ldots, k \} \) by relabeling the coordinate axes on \( \RR^{n} \). The Lagrangian \( \lambda_S = \{(q, Sq)\} \) is realized as the column space of the matrix
\( \begin{bmatrix} I & S \end{bmatrix}^{\mathsf{T}} \),
and therefore \( J_K\lambda_S = \varphi_K(S) \) is the column space of
\begin{equation} \label{columnspace}
J_K\begin{bmatrix}
I \\
S
\end{bmatrix} = J_K\begin{bmatrix}
I_{k \times k} & 0 \\
0 & I_{(n-k) \times (n-k)}\\
S_1 & S_2 \\
S_3 & S_4
\end{bmatrix} = \begin{bmatrix}
-S_1 & -S_2 \\
0 & I \\
I & 0 \\
S_3 & S_4 \\
\end{bmatrix}.
\end{equation}

Since the column vectors are linearly independent, the column space of the last \( n-k \) vectors always has trivial intersection with \( \sigma \). If \( S_1 = 0 \), then the first \( k \) column vectors form a basis for the intersection \( \sigma \cap \lambda \), and conversely, if \( \dim(\lambda \cap \sigma) = k \), it must be the case that \( S_1 = 0 \).
\end{proof}

More generally, we see that, for \( l \leq k \) and \( \lambda \in \Lambda^0(\sigma_K) \),
\[
\lambda \in \Lambda^{l}(\sigma) \iff \dim \ker S_1 = l.
\]

This shows that every \( \lambda \in \Lambda \) belongs to some chart \( \varphi_K(\Sym(n,\RR)) \) for some \( K \subseteq \{ 1, \ldots, n \} \), so they cover \( \Lambda \). It is easy to see that they are compatible, indeed showing that these maps form an atlas for an embedded submanifold of the Grassmannian of \( n \)-planes of \( \RR^{2n} \). We also conclude that the subsets \( \Lambda^k(\sigma) \) form embedded submanifolds of codimension \( k(k+1)/2 \).

\section{The Intersection Number}

From the original question of understanding curves \( \lambda(t) \) of lagrangian subspaces and when does \( \dim(\lambda(t) \cap \sigma) > 0 \), we are naturally led to consider intersections of \( \lambda(t) \) with the set \( \Lambda^{\geq 1}(\sigma) = \bigcup_{k \geq 1}\Lambda^k(\sigma) \). We will show that the intersection number of an oriented curve with this subset \cite{guillemin}*{Chapter 3} is well defined, as \( \Lambda^{\geq 1}(\sigma) \) is a two-sided cycle of codimension \( 1 \), and we relate this index of intersection to a canonical cohomology class with integer coefficients in order to provide explicit calculations.

The unitary group \( \mathrm{U}(n) \) acts smoothly and transitively on the Lagrangian subspaces of \( \RR^{2n} \), and the isotropy subgroup of \( \RR^n \times \{0\} \cong \RR^n \subset \CC^n \) is the orthogonal group \( \mathrm{O}(n) \). This implies that the Lagrangian Grasmmannian can be realized as the homogeneous manifold \( \mathrm{U}(n)/\mathrm{O}(n) \).

The map \( \det^2 : ~\mathrm{U}(n) \to S^1 \) is well defined on the quotient, resulting on the induced map
\begin{equation} \label{Det^2}
\Det^2: \mathrm{U}(n)/\mathrm{O}(n) \longrightarrow S^1.
\end{equation}

\begin{prop} \label{formula}
For \( \lambda_S \in \Lambda^0(\sigma) \), we have that
\begin{equation} \label{formulaforDet^2}
\Det^2 \lambda_S = \det\dfrac{I + iS}{I - iS}.
\end{equation}
\end{prop}

\begin{proof}
With the identification \( \RR^{2n} \cong \CC^n \), the map \( I + iS \) takes the Lagrangian subspace \( \RR^n = \{(q,0)\} \) to \( \lambda_S = \{(q, Sq)\} \). It may not be unitary, but \( (I + iS)/\sqrt{I + S^2} \) is; and since \( \sqrt{I + S^2} \) preserves \( \RR^n \), this map also takes \( \RR^n \) to \( \lambda_S \). Therefore
\[
\Det^2 \lambda_S = \det{}^2\left(\dfrac{I +i S}{\sqrt{I + S^2}}\right) = \det \dfrac{(I + iS)^2}{I + S^2} = \det \dfrac{I + iS}{I - iS}.
\]
\end{proof}

Consider also the set \( \mathrm{S}\Lambda(n) \) of all Lagrangian subspaces \( \lambda \) such that \( \Det^2\lambda = 1 \). Then \( \mathrm{SU}(n) \) acts transitively on \( \mathrm{S}\Lambda(n) \) with stabilizer \( \mathrm{SO}(n) \), so that \( \mathrm{S}\Lambda(n) \cong \mathrm{SU}(n)/\mathrm{SO}(n) \).

The map \( \Det^2 \) in fact induces an isomorphism \( \pi_1(\Lambda) \cong \pi_1(S^1) \) between the fundamental groups. This can be seen through the exact homotopy sequences of the six fibrations of the following commutative diagram:
\[
\xymatrix{
\mathrm{SO}(n) \ar[r] \ar[d] & \mathrm{O}(n) \ar[r]^{\det} \ar[d] & S^0 \ar[d] \\
\mathrm{SU}(n) \ar[r] \ar[d] & \mathrm{U}(n) \ar[r]^{\det} \ar[d] & S^1 \ar[d]^{z \mapsto z^2} \\
\mathrm{S}\Lambda(n) \ar[r] & \Lambda(n) \ar[r]^{\Det^2} & S^1
}
\]
More explicitly, \( \mathrm{S}\Lambda(n) \) and \( \Lambda(n) \) are both connected, being continuous images of \( \mathrm{SU}(n) \) and \( \mathrm{U}(n) \), the long exact sequence
\[
\cdots \to \pi_1(\mathrm{SO}(n)) \to \pi_1(\mathrm{SU}(n)) \to \pi_1(\mathrm{S}\Lambda(n)) \to \pi_0(\mathrm{SO}(n)) \to \cdots
\]
gives us \( \pi_1(\mathrm{S}\Lambda(n)) = 0 \), and the long exact sequence
\[
\cdots \to \pi_1(\mathrm{S}\Lambda(n)) \to \pi_1(\Lambda(n)) \to \pi_1(S^1) \to \pi_0(\mathrm{S}\Lambda(n)) \to \cdots
\]
gives us the aforementioned isomorphism.

Recall that \( \deg: \pi_1(S^1) \to \ZZ \) is an isomorphism, and in this case, the Hurewicz map \( \pi_1(\Lambda) \to H_1(\Lambda; \ZZ) \) given by the abelianization of the fundamental group is also an isomorphism. This allows us to conclude that
\[
H^1(\Lambda ; \ZZ) \cong \Hom(H_1(\Lambda; \ZZ), \ZZ) \cong H_1(\Lambda; \ZZ) \cong \ZZ,
\]
as \( \ZZ \) is abelian. Consider \( \alpha \in H^1(\Lambda ; \ZZ) \cong \Hom(\pi_1(\Lambda), \ZZ) \) to be the cohomology class given by
\begin{equation} \label{maslovindex}
\alpha(\gamma) = \deg(\Det^2 \circ \gamma),
\end{equation}
where \( \gamma \) is a closed curve given up to homotopy. Then \( \alpha \) coincides with the pullback of the angle \( 1 \)-form \( d\theta \) on \( S^1 \) by \( \Det^2 \), where \( \alpha \) is evaluated on smooth closed curves belonging to the same homotopy class. In certain contexts \( \alpha \) is referred to as the Maslov index of the Lagrangian Grassmannian \( \Lambda \). It is readily verifiable that \( \alpha \) is a generator for \( H^1(\Lambda ; \ZZ) \) through the following diagram of isomorphisms:
\[
\xymatrix{
\pi_1(\Lambda) \ar[r]^{\det^2_*} \ar[dr]^{\alpha} & \pi_1(S^1) \ar[d]^{\deg} \\
& \ZZ
}
\]

Fixing \( \sigma = \{0\} \times \RR^n \) as before, we shall prove that \( \alpha \) is equal to the index of intersection of an oriented closed curve with \( \Lambda^{\geq 1}(\sigma) \), that is, the set of Lagrangians which have non-trivial intersection with \( \sigma \).

Note that \( \Lambda(n) \) can be regarded as an algebraic manifold, so that the closure \( \overline{\Lambda^1(\sigma)} \), being equal to the union \( \bigcup_{k = 1}^n \Lambda^k(\sigma) = \Lambda^{\geq 1}(\sigma) \), determines an algebraic submanifold of codimension \( 1 \). Since the higher strata of \( \Lambda^{\geq 1}(\sigma) \) correspond to the boundary \( \partial \Lambda^{\geq 1}(\sigma) = \bigcup_{k = 2}^n \Lambda^k(\sigma) = \Lambda^{\geq 2}(\sigma) \), this singularity is of codimension \( 2(2+1)/2 = 3 \) in \( \Lambda(n) \), which means that the homological boundary of \( \overline{\Lambda^1(\sigma)} \) is \( 0 \). Consequently, \( \overline{\Lambda^1(\sigma)} \) is a cycle of codimension \( 1 \).

\begin{lema}
\( \overline{\Lambda^1(\sigma)} \) is a two-sided cycle in \( \Lambda(n) \).
\end{lema}

\begin{proof}
We must show that there exists a non-vanishing continuous vector field along \( \Lambda^1(\sigma) \) transversal to it. The flow \( \lambda \mapsto e^{it}\lambda \) for \( t \in \RR \) on \( \Lambda \) produces an infinitesimal generator which, along \( \Lambda^1(\sigma) \), will be the desired vector field. On a chart \( \varphi_K(\Sym(n,\RR)) \) we have
\[
\lambda(t) = J_K \lambda_{S(t)} = e^{it} J_K\lambda_{S(0)} \implies \lambda_{S(t)} = e^{it}\lambda_{S(0)}.
\]
This means that the column vectors of the \( 2n \times n \) matrices
\[
\begin{bmatrix}
I \\
S(t)
\end{bmatrix}, \quad  \begin{bmatrix}
\cos t I & -\sin t I \\
\sin t I & \cos t I
\end{bmatrix} \begin{bmatrix}
I \\
S(0)
\end{bmatrix}
\]
span the same subspace, hence there exists a curve \( G(t) \in \mathrm{GL}(n,\RR) \) such that
\[
\begin{bmatrix}
I \\
S(t)
\end{bmatrix} = \begin{bmatrix}
\cos t I & -\sin t I \\
\sin t I & \cos t I
\end{bmatrix} \begin{bmatrix}
I \\
S(0)
\end{bmatrix} G(t).
\]
This in turn implies
\begin{equation} \label{flow}
S(t) = \dfrac{\sin t I + \cos t S(0)}{\cos t I - \sin t S(0)},
\end{equation}
so that \( S'(0) = I + S(0)^2 = I + S(0)S(0)^t \). If \( \lambda(0) \in \Lambda^1(\sigma) \) and \( K = \{ \kappa \} \), then \( S'_{\kappa \kappa}(0) \geq 1 \), so that the flow is indeed transversal to \( \Lambda^1(\sigma) \). Hence \( \Lambda^{\geq 1}(\sigma) \) is two-sided and a positive orientation can be given by the flow \( e^{it}\lambda \).
\end{proof}

With this, we can properly define the index of intersection \( \Ind(\gamma) \) of an oriented curve \( \gamma: [a,b] \to \Lambda \) with the cycle \( \Lambda^{\geq 1}(\sigma) \) when \( \gamma(a), \gamma(b) \in \Lambda^0(\sigma) \), which is invariant up to homotopy of \( \gamma \) fixing its endpoints. This is because we can complete \( \gamma \) to a closed curve by joining \( \gamma(b) \) to \( \gamma(a) \) through any path in \( \Lambda^0(\sigma) \), since it is a simply connected open set.

We show that the index of intersection and \( \alpha \) coincide:

\begin{prop}
\( \Ind(\gamma) = \alpha(\gamma) \) for all \( [\gamma] \in \pi_1(\Lambda) \).
\end{prop}

\begin{proof}
It suffices to prove the equality for a specific closed curve, since \( \alpha \) is a generator for \( H^1(\Lambda; \ZZ) \). We take \( \gamma \) to be the closed curve \( e^{it}\lambda \) for \( 0 \leq t \leq \pi \), where \( \lambda \in \Lambda \) is to be chosen. For almost all \( \lambda \in \Lambda \), the curve \( e^{it}\lambda \) does not pass through \( \overline{\Lambda^2(\sigma)} \). So we may take such \( \lambda = \lambda_S \in \Lambda^0(\sigma) \) and such that \( S \) has nonzero and pairwise distinct eigenvalues. Then \( S(0) = S(\pi) = S \), and at the points where \( e^{it}\lambda \) intersects \( \Lambda^1(\sigma) \), it will do so transversally and positively. By (\ref{flow}), these points of intersection correspond to the values of \( t \in (0, \pi) \) for which
\[
\det(\cos tI - \sin tS) = (-\sin t)^n \det(S - \cot t I) = 0.
\]

For \( t \in (0, \pi) \), \( \cot t \) parametrizes \( \RR \) once, so the determinant vanishes exactly for the \( n \) distinct real eigenvalues of \( S \). This means that \( \Ind(\gamma) = n \). As for \( \alpha(\gamma) \), we have
\[
\Det^2 e^{it}\lambda_S = e^{2nit}\Det^2 \lambda_S,
\]
which winds around the circle \( n \) times for \( 0 \leq t \leq \pi \). So \( \alpha(\gamma) = n = \Ind(\gamma) \), and \( \alpha = \Ind \) for general closed curves.
\end{proof}

\section{On The Symplectic Flow}

We return to the curves on \( \Lambda(n) \) given by a symplectic flow of the form of the Jacobi equation, so that we may calculate their intersection numbers with \( \Lambda^{\geq 1}(\sigma) \). It is important to note that these intersections will not in general be transversal, that is, ocurring transversally at the principal stratum \( \Lambda^1(\sigma) \); and even though we can perturb the curve to a homotopic one that does so, we may not necessarily know the information on the multiplicities of the intersections. Fortunately, these intersections will still be non-degenerate in a precise sense, where we can adequately describe their contributions to the index of intersection.

For \( t \in [a,b] \) and any Lagrangian subspace \( \tau \in \Lambda \), let
\begin{equation}
\lambda(t) \coloneqq \{ v(t) \in \RR^{2n} \mid v(a) \in \tau, \ v'(t) = A(t)v(t) \},   
\end{equation}
where
\[
A(t) = \begin{bmatrix}
0 & I \\
R(t) & 0
\end{bmatrix} \in \mathfrak{sp}(2n), \ R(t) \in \Sym(n,\RR).
\]

\begin{lema}
\( \lambda(t) \) intersects \( \Lambda^{\geq 1}(\sigma) \) finitely many times.
\end{lema}

\begin{proof}
Suppose that, at a time \( t_0 \in [a,b] \), we have that \( \dim(\lambda(t_0) \cap \sigma) = k \), that is, \( \lambda(t_0) \in \Lambda^k(\sigma) \). Then by lemma \ref{chart} there exists \( K \subseteq \{ 1, \ldots, n \} \), which we may assume to be \( \{1, \ldots, k \} \), such that \( \lambda(t_0) \in \Lambda^0(\sigma_K) \). Since it is an open set, we may assume also that for \( t \) close to \( t_0 \) we have \( \lambda(t) \) contained in the chart \( \Lambda^0(\sigma_K) \). We define the curve of matrices \( S(t) = \varphi_K^{-1}\lambda(t) \in \Sym(n,\RR) \), where explicitly \(
\lambda(t) = J_K \lambda_{S(t)} = \colsp L(t) \) for
\[
L(t) = \begin{bmatrix}
-S_1(t) & -S_2(t) \\
0 & I \\
I & 0 \\
S_3(t) & S_4(t) \\
\end{bmatrix},
\]
and such that \( S_1(t_0) = 0_{k \times k} \). Let \( \Phi(t) \) be the fundamental matrix for the flow \( v'(t) = A(t)v(t) \) such that
\begin{equation}
\begin{cases}
\Phi'(t) = A(t) \Phi(t), \\
\Phi(t_0) = I_{2n \times 2n};
\end{cases}
\end{equation}
then
\[
\gamma(t) = \colsp (L(t)) = \colsp \left( \Phi(t) L(t_0) \right) = \Phi(t)\gamma(t_0).
\]

Since \( L(t) \) and \( \Phi(t)L(t_0) \) have the same column space, there exists a curve \( G(t) \) in \( \mathrm{GL}_n(\RR) \) such that \( G(t_0) = I_{n \times n} \) and
\begin{equation}
L(t) = \Phi(t)L(t_0)G(t).
\end{equation}
We differentiate the expression above at \( t = t_0 \):
\begin{align}
L'(t_0) & = \Phi'(t_0)L(t_0)G(t_0) + \Phi(t_0)L(t_0)G'(t_0) \nonumber \\
& = A(t_0)L(t_0) + L(t_0)G'(t_0). \label{columnspaces}
\end{align}

By also considering
\begin{gather*}
R(t) = \begin{bmatrix}
R_1(t)_{k \times k} & R_2(t)_{k \times (n-k)} \\
R_3(t)_{(n-k) \times k} & R_4(t)_{(n-k) \times (n-k)}
\end{bmatrix}, \\
G(t) = \begin{bmatrix}
G_1(t)_{k \times k} & G_2(t)_{k \times (n-k)} \\
G_3(t)_{(n-k) \times k} & G_4(t)_{(n-k) \times (n-k)}
\end{bmatrix},
\end{gather*}
we may expand the equation (\ref{columnspaces}) in matrix form:
\[
\begin{bmatrix}
-S_1' & -S_2' \\
0 & 0 \\
S_3' & S_4' \\
0 & 0
\end{bmatrix} = \begin{bmatrix}
I - S_2G_3' & -S_2G_4' \\
S_3 + G_3' & S_4 + G_4' \\
G_1' & -R_1S_2 + R_2 + G_2' \\
S_3G_1' + S_4G_3' & -R_3S_2 + R_4 + S_3G_2' + S_4G_4'
\end{bmatrix},
\]
where all the matrices above are evaluated at \( t_0 \). Then \( G_3'(t_0) = -S_3(t_0) = -S_2(t_0)^{\mathsf{T}} \) and
\begin{equation*}
S_1'(t_0) = - I - S_2(t_0)S_2(t_0)^{\mathsf{T}},
\end{equation*}
which is negative definite, and all of its eigenvalues are \( \leq -1 \). Since \( S_1(t_0) = 0 \), this implies that for \( t \) close to \( t_0 \) and \( t < t_0 \), the eigenvalues of \( S_1(t) \) are all positive, and for \( t > t_0 \), they are all negative. In particular, for \( t \neq t_0 \), \( S_1(t) \) has trivial kernel, so \( \lambda(t) \in \Lambda^0(\sigma) \). Thus intersections of the curve with \( \Lambda^{\geq 1}(\sigma) \) are discrete, and since the interval is compact, they are finite.
\end{proof}

\begin{teo} \label{indexofintersection}
On the conditions of the previous lemma, if we also assume that \( \lambda(a), \lambda(b) \in \Lambda^0(\sigma) \), then
\begin{equation}
\Ind(\lambda) = \sum_{t \in (a,b)} \dim(\lambda(t) \cap \sigma),
\end{equation}
where the sum above has finitely many non-zero terms.
\end{teo}

\begin{proof}
It suffices to show that at each intersection of \( \lambda \) with \( \Lambda^{\geq 1}(\sigma) \), the contribution to the intersection number is given by the \( k \) such that the intersection is on the stratum \( \Lambda^{k}(\sigma) \). As before, if \( \lambda(t_0) \in \Lambda^{k}(\sigma) \), then locally \( \lambda(t) \in \Lambda^0(\sigma_K) \) for \( K \subseteq \{ 1, \ldots, n \} \), which we may assume to be \( \{ 1, \ldots, k \} \), and \( S(t) = \varphi_K^{-1}(\lambda(t)) \).

Since, for \( t_1 < t_0 \), \( S(t_1) \in \Sym(n,\RR) \cap \varphi_K^{-1}(\Lambda^0(\sigma)) \) and all its eigenvalues are positive, we can find a path joining \( S(t_1) \) to the matrix
\[
E_1 = \begin{bmatrix}
I & 0 \\
0 & 0
\end{bmatrix}
\]
that avoids \( \varphi_K^{-1}(\Lambda^{\geq 1}(\sigma)) \). This is done considering a diagonalization \( S_1(t_1) = MDM^{-1} \), where \( M \in \mathrm{O}(n) \) and \( D \) is diagonal with positive eigenvalues, and simultaneously deforming \( D \) to the above matrix and \( M \) either to the identity or to a simple reflection about the \( x^1 \) axis, whether the determinant of \( M \) is \( 1 \) or \( -1 \). All the other entries are taken to be \( 0 \) through a linear homotopy. Similarly, for \( t_2 > t_0 \), we can find a path joining \( S(t_2) \) to the matrix
\[
E_2 = -E_1 = \begin{bmatrix}
-I & 0 \\
0 & 0
\end{bmatrix}
\]
which avoids \( \varphi_K^{-1}(\Lambda^{\geq 1}(\sigma)) \).

Now we consider the curve \( \eta: [-1,1] \to \Lambda(n) \) given by \( \eta(t) = \varphi_K(T(t)) \), where
\begin{equation}  \label{parametrization}
T(t) = \begin{bmatrix}
-tI & 0 \\
0 & 0
\end{bmatrix}.
\end{equation}
This curve will intersect \( \Lambda^{\geq 1}(\sigma) \) only at the value \( t = 0 \). Furthermore, the Lagrangian subspaces at the endpoints of the curve are
\[
\eta(-1) = J_K \colsp \begin{bmatrix}
I & 0 \\
0 & I \\
I & 0 \\
0 & 0
\end{bmatrix} = \colsp \begin{bmatrix}
-I & 0 \\
0 & I \\
I & 0 \\
0 & 0
\end{bmatrix} = \colsp \begin{bmatrix}
I & 0 \\
0 & I \\
-I & 0 \\
0 & 0
\end{bmatrix} = \varphi(E_2),
\]
and analogously \( \eta(1) = \varphi(E_1) \). We connect the Lagrangian subspaces \( \eta(1) \) and \( \eta(-1) \) through the same parametrization in (\ref{parametrization}),
but on \( \varphi^{-1}(\Lambda^0(\sigma)) \), a different chart. The curve \( \mu(t) = \varphi(T(t)) \) defined on \( [-1,1] \) is such that \( \mu(-1) = \eta(1) \), \( \mu(1) = \eta(-1) \) and \( \mu \) is contained in \( \Lambda^0(\sigma) \). Concatenating both curves at their endpoints, we form a simple closed curve that intersects \( \Lambda^{\geq 1}(\sigma) \) at only one point, and we can use the parametrizations to calculate its index of intersection. Explicitly, by the formula in proposition \ref{formula},
\[
\Det^2 \eta(t) = \Det^2 J_K \lambda_{T(t)} = i^{2k} \left(\dfrac{1 - it}{1 + it}\right)^k,
\]
which winds around the circle \( k/2 \) times for \( t \in [-1,1] \), and
\[
\Det^2 \mu(t) = \Det^2 \lambda_{T(t)} = \left(\dfrac{1-it}{1+it}\right)^k,
\]
which further winds around the circle \( k/2 \) times. Then \( \alpha(\mu*\eta) = k \), and it coincides with the index of intersection with \( \Lambda^{\geq 1}(\sigma) \). More importantly, it does not depend on how we complete the curve \( \eta \) through \( \Lambda^0(\sigma) \), and is invariant under homotopies.

Since \( \Lambda^0(\sigma_K) \) is simply connected, we can find a homotopy between the curve \( \lambda(t) \) on \( [t_1, t_2] \) and \( \eta \) which preserves the intersection number, as the path joining \( \lambda(t_i) \) to \( \eta(E_i) \) does not intersect \( \Lambda^{\geq 1}(\sigma) \). Finally, this implies that each intersection point \( \lambda(t_0) \in \Lambda^k(\sigma) \) contributes exactly \( k \) to the intersection number of the curve \( \lambda(t) \) with \( \Lambda^{\geq 1}(\sigma) \).
\end{proof}

\section{Final Steps}

We return to the Lagrangian subspaces \( \sigma_{\lambda}(t) \) as defined in (\ref{sigmalambda}) in order to prove the equality (\ref{homologicalrectangle}). We may consider \( \lambda \in [\lambda_0, 0] \) for some \( \lambda_0 < 0 \), where for \( \mu \leq \lambda_0 \), we have \( \sigma_{\mu}(t) \cap \sigma = \{0\} \). The map \( \sigma_{\lambda}(t) \) is smooth on both variables, and the image of \( \sigma : [a,b] \times [\lambda_0, 0] \) forms a homological rectangle:

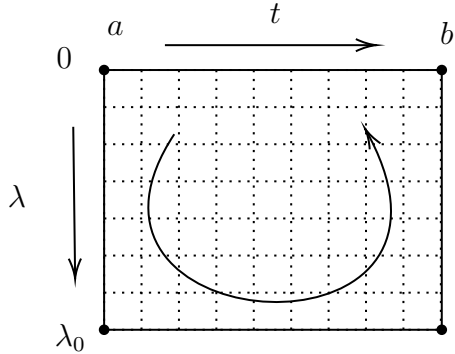
\begin{center}

\tikzset{every picture/.style={line width=0.75pt}} 

\begin{tikzpicture}[x=0.7pt,y=0.7pt,yscale=-1,xscale=1]

\draw    (226.5,75.6) .. controls (218.24,88.22) and (214.01,99.63) .. (212.93,109.75) .. controls (204.29,190.33) and (394.24,189.86) .. (328.5,72.6) ;
\draw [shift={(328.5,72.6)}, rotate = 60.72] [color={rgb, 255:red, 0; green, 0; blue, 0 }  ][line width=0.75]    (10.93,-3.29) .. controls (6.95,-1.4) and (3.31,-0.3) .. (0,0) .. controls (3.31,0.3) and (6.95,1.4) .. (10.93,3.29)   ;
\draw  [draw opacity=0][dash pattern={on 0.84pt off 2.51pt}] (189,41) -- (369.5,41) -- (369.5,182.6) -- (189,182.6) -- cycle ; \draw  [dash pattern={on 0.84pt off 2.51pt}] (189,41) -- (189,182.6)(209,41) -- (209,182.6)(229,41) -- (229,182.6)(249,41) -- (249,182.6)(269,41) -- (269,182.6)(289,41) -- (289,182.6)(309,41) -- (309,182.6)(329,41) -- (329,182.6)(349,41) -- (349,182.6)(369,41) -- (369,182.6) ; \draw  [dash pattern={on 0.84pt off 2.51pt}] (189,41) -- (369.5,41)(189,61) -- (369.5,61)(189,81) -- (369.5,81)(189,101) -- (369.5,101)(189,121) -- (369.5,121)(189,141) -- (369.5,141)(189,161) -- (369.5,161)(189,181) -- (369.5,181) ; \draw  [dash pattern={on 0.84pt off 2.51pt}]  ;
\draw   (189,41) -- (369.5,41) -- (369.5,181) -- (189,181) -- cycle ;
\draw    (172.5,71.6) -- (173.48,151.6) ;
\draw [shift={(173.5,153.6)}, rotate = 269.3] [color={rgb, 255:red, 0; green, 0; blue, 0 }  ][line width=0.75]    (10.93,-3.29) .. controls (6.95,-1.4) and (3.31,-0.3) .. (0,0) .. controls (3.31,0.3) and (6.95,1.4) .. (10.93,3.29)   ;
\draw    (221.5,27.6) -- (332.5,27.6) ;
\draw [shift={(334.5,27.6)}, rotate = 180] [color={rgb, 255:red, 0; green, 0; blue, 0 }  ][line width=0.75]    (10.93,-3.29) .. controls (6.95,-1.4) and (3.31,-0.3) .. (0,0) .. controls (3.31,0.3) and (6.95,1.4) .. (10.93,3.29)   ;
\draw  [fill={rgb, 255:red, 0; green, 0; blue, 0 }  ,fill opacity=1 ] (186.7,41) .. controls (186.7,39.73) and (187.73,38.7) .. (189,38.7) .. controls (190.27,38.7) and (191.3,39.73) .. (191.3,41) .. controls (191.3,42.27) and (190.27,43.3) .. (189,43.3) .. controls (187.73,43.3) and (186.7,42.27) .. (186.7,41) -- cycle ;
\draw  [fill={rgb, 255:red, 0; green, 0; blue, 0 }  ,fill opacity=1 ] (186.7,181) .. controls (186.7,179.73) and (187.73,178.7) .. (189,178.7) .. controls (190.27,178.7) and (191.3,179.73) .. (191.3,181) .. controls (191.3,182.27) and (190.27,183.3) .. (189,183.3) .. controls (187.73,183.3) and (186.7,182.27) .. (186.7,181) -- cycle ;
\draw  [fill={rgb, 255:red, 0; green, 0; blue, 0 }  ,fill opacity=1 ] (367.2,41) .. controls (367.2,39.73) and (368.23,38.7) .. (369.5,38.7) .. controls (370.77,38.7) and (371.8,39.73) .. (371.8,41) .. controls (371.8,42.27) and (370.77,43.3) .. (369.5,43.3) .. controls (368.23,43.3) and (367.2,42.27) .. (367.2,41) -- cycle ;
\draw  [fill={rgb, 255:red, 0; green, 0; blue, 0 }  ,fill opacity=1 ] (367.2,181) .. controls (367.2,179.73) and (368.23,178.7) .. (369.5,178.7) .. controls (370.77,178.7) and (371.8,179.73) .. (371.8,181) .. controls (371.8,182.27) and (370.77,183.3) .. (369.5,183.3) .. controls (368.23,183.3) and (367.2,182.27) .. (367.2,181) -- cycle ;

\draw (189,13.4) node [anchor=north west][inner sep=0.75pt]    {$a$};
\draw (367,13.4) node [anchor=north west][inner sep=0.75pt]    {$b$};
\draw (276,3.4) node [anchor=north west][inner sep=0.75pt]    {$t$};
\draw (162,27.4) node [anchor=north west][inner sep=0.75pt]    {$0$};
\draw (136,101.4) node [anchor=north west][inner sep=0.75pt]    {$\lambda $};
\draw (161,176.4) node [anchor=north west][inner sep=0.75pt]    {$\lambda _{0}$};

\end{tikzpicture}
\end{center}

In principle, we know how to calculate the intersection number of the horizontal and vertical sides of this rectangle with \( \Lambda^{\geq 1}(\sigma) \), since they are given by a symplectic flow, varying either the term \( R(t) \) or \( -\lambda I \) in (\ref{secondorderode}). However, \( \sigma_0(a) \notin \Lambda^{0}(\sigma) \), and possibly \( \sigma_0(b) \notin \Lambda^{0}(\sigma) \). To proceed, we consider curves homotopic to these edges for which we can apply theorem \ref{indexofintersection}.

Note that \( \sigma_{\lambda}(a) = \sigma \) for all \( \lambda \in [\lambda_0, 0] \). Since \( \sigma \in \Lambda^n(\sigma) \), we have that \( \sigma_{\lambda}(t) \in \Lambda^{0}(\sigma_N) \) for all \( \lambda \) and for \( t \) close to \( a \), where \( N = \{1, \ldots, n\} \). Then \( \sigma_{\lambda}(t) = J_N \lambda_{S_{\lambda(t)}} = J\lambda_{S_{\lambda(t)}} \), and for \( t \) closer to \( a \), the matrices \( S_{\lambda}(t) \) all have negative eigenvalues. This means that we can find \( a' \in (a,b) \) such that, for \( t \in (a,a'] \) and all \( \lambda \), we have \( \sigma_{\lambda}(t) \in \Lambda^0(\sigma) \). The edge \( \sigma_{\cdot}(a) \equiv \sigma \) is then homotopic to \( \sigma_{\cdot}(a') \) and has the same intersection number, which is \( 0 \).

Similarly, if \( \sigma_{\lambda}(b) \in \Lambda^k(\sigma) \) for some \( k \geq 1 \), then for \( t \) close to \( b \) we have \( \sigma_{0}(t) \in \Lambda^0(\sigma) \), and in the chart which \( \sigma_0(b) \) belongs to, the corresponding \( k \times k \) matrix has all negative eigenvalues. This is the same for \( \sigma_{\lambda}(b) \) when \( \lambda \) is close to \( 0 \), so we may take \( b' \in (a,b) \) and \( \lambda' \in (\lambda_0, 0) \) such that \( \sigma_{0}(b') \) and \( \sigma_{\lambda'}(b) \) are joined by a homotopic path in \( \Lambda^0(\sigma) \).

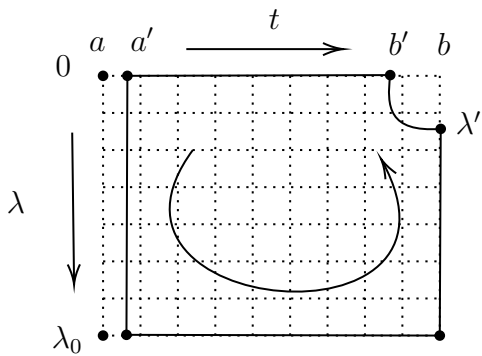
\begin{center}

\tikzset{every picture/.style={line width=0.75pt}} 

\begin{tikzpicture}[x=0.7pt,y=0.7pt,yscale=-1,xscale=1]

\draw  [draw opacity=0][dash pattern={on 0.84pt off 2.51pt}] (227,46) -- (408.5,46) -- (408.5,187.6) -- (227,187.6) -- cycle ; \draw  [dash pattern={on 0.84pt off 2.51pt}] (227,46) -- (227,187.6)(247,46) -- (247,187.6)(267,46) -- (267,187.6)(287,46) -- (287,187.6)(307,46) -- (307,187.6)(327,46) -- (327,187.6)(347,46) -- (347,187.6)(367,46) -- (367,187.6)(387,46) -- (387,187.6)(407,46) -- (407,187.6) ; \draw  [dash pattern={on 0.84pt off 2.51pt}] (227,46) -- (408.5,46)(227,66) -- (408.5,66)(227,86) -- (408.5,86)(227,106) -- (408.5,106)(227,126) -- (408.5,126)(227,146) -- (408.5,146)(227,166) -- (408.5,166)(227,186) -- (408.5,186) ; \draw  [dash pattern={on 0.84pt off 2.51pt}]  ;
\draw    (210.5,76.6) -- (211.48,156.6) ;
\draw [shift={(211.5,158.6)}, rotate = 269.3] [color={rgb, 255:red, 0; green, 0; blue, 0 }  ][line width=0.75]    (10.93,-3.29) .. controls (6.95,-1.4) and (3.31,-0.3) .. (0,0) .. controls (3.31,0.3) and (6.95,1.4) .. (10.93,3.29)   ;
\draw    (271.5,31.6) -- (353.5,31.6) ;
\draw [shift={(355.5,31.6)}, rotate = 180] [color={rgb, 255:red, 0; green, 0; blue, 0 }  ][line width=0.75]    (10.93,-3.29) .. controls (6.95,-1.4) and (3.31,-0.3) .. (0,0) .. controls (3.31,0.3) and (6.95,1.4) .. (10.93,3.29)   ;
\draw    (240,46) -- (239.5,185.6) ;
\draw    (407,186) -- (239.5,185.6) ;
\draw    (407.5,74.6) -- (407,186) ;
\draw    (240,46) -- (380.5,45.6) ;
\draw    (380.5,45.6) .. controls (378.5,62.6) and (379.5,76.6) .. (407.5,74.6) ;
\draw    (275,86) .. controls (208.83,176.15) and (432.25,196.4) .. (376.36,92.18) ;
\draw [shift={(375.5,90.6)}, rotate = 60.9] [color={rgb, 255:red, 0; green, 0; blue, 0 }  ][line width=0.75]    (10.93,-3.29) .. controls (6.95,-1.4) and (3.31,-0.3) .. (0,0) .. controls (3.31,0.3) and (6.95,1.4) .. (10.93,3.29)   ;
\draw  [fill={rgb, 255:red, 0; green, 0; blue, 0 }  ,fill opacity=1 ] (237.7,46) .. controls (237.7,44.73) and (238.73,43.7) .. (240,43.7) .. controls (241.27,43.7) and (242.3,44.73) .. (242.3,46) .. controls (242.3,47.27) and (241.27,48.3) .. (240,48.3) .. controls (238.73,48.3) and (237.7,47.27) .. (237.7,46) -- cycle ;
\draw  [fill={rgb, 255:red, 0; green, 0; blue, 0 }  ,fill opacity=1 ] (237.2,185.6) .. controls (237.2,184.33) and (238.23,183.3) .. (239.5,183.3) .. controls (240.77,183.3) and (241.8,184.33) .. (241.8,185.6) .. controls (241.8,186.87) and (240.77,187.9) .. (239.5,187.9) .. controls (238.23,187.9) and (237.2,186.87) .. (237.2,185.6) -- cycle ;
\draw  [fill={rgb, 255:red, 0; green, 0; blue, 0 }  ,fill opacity=1 ] (378.2,45.6) .. controls (378.2,44.33) and (379.23,43.3) .. (380.5,43.3) .. controls (381.77,43.3) and (382.8,44.33) .. (382.8,45.6) .. controls (382.8,46.87) and (381.77,47.9) .. (380.5,47.9) .. controls (379.23,47.9) and (378.2,46.87) .. (378.2,45.6) -- cycle ;
\draw  [fill={rgb, 255:red, 0; green, 0; blue, 0 }  ,fill opacity=1 ] (405.2,74.6) .. controls (405.2,73.33) and (406.23,72.3) .. (407.5,72.3) .. controls (408.77,72.3) and (409.8,73.33) .. (409.8,74.6) .. controls (409.8,75.87) and (408.77,76.9) .. (407.5,76.9) .. controls (406.23,76.9) and (405.2,75.87) .. (405.2,74.6) -- cycle ;
\draw  [fill={rgb, 255:red, 0; green, 0; blue, 0 }  ,fill opacity=1 ] (404.7,186) .. controls (404.7,184.73) and (405.73,183.7) .. (407,183.7) .. controls (408.27,183.7) and (409.3,184.73) .. (409.3,186) .. controls (409.3,187.27) and (408.27,188.3) .. (407,188.3) .. controls (405.73,188.3) and (404.7,187.27) .. (404.7,186) -- cycle ;
\draw  [fill={rgb, 255:red, 0; green, 0; blue, 0 }  ,fill opacity=1 ] (224.7,46) .. controls (224.7,44.73) and (225.73,43.7) .. (227,43.7) .. controls (228.27,43.7) and (229.3,44.73) .. (229.3,46) .. controls (229.3,47.27) and (228.27,48.3) .. (227,48.3) .. controls (225.73,48.3) and (224.7,47.27) .. (224.7,46) -- cycle ;
\draw  [fill={rgb, 255:red, 0; green, 0; blue, 0 }  ,fill opacity=1 ] (224.7,186) .. controls (224.7,184.73) and (225.73,183.7) .. (227,183.7) .. controls (228.27,183.7) and (229.3,184.73) .. (229.3,186) .. controls (229.3,187.27) and (228.27,188.3) .. (227,188.3) .. controls (225.73,188.3) and (224.7,187.27) .. (224.7,186) -- cycle ;

\draw (218,23.4) node [anchor=north west][inner sep=0.75pt]    {$a$};
\draw (404,21.4) node [anchor=north west][inner sep=0.75pt]    {$b$};
\draw (314,8.4) node [anchor=north west][inner sep=0.75pt]    {$t$};
\draw (200,33.4) node [anchor=north west][inner sep=0.75pt]    {$0$};
\draw (174,106.4) node [anchor=north west][inner sep=0.75pt]    {$\lambda $};
\draw (198,179.4) node [anchor=north west][inner sep=0.75pt]    {$\lambda _{0}$};
\draw (378,20.4) node [anchor=north west][inner sep=0.75pt]    {$b'$};
\draw (414,63.4) node [anchor=north west][inner sep=0.75pt]    {$\lambda '$};
\draw (240,18.4) node [anchor=north west][inner sep=0.75pt]    {$a'$};

\end{tikzpicture}
\end{center}

This new loop \( \eta \) is still contractible, so \( \alpha(\eta) = \Ind(\eta) = 0 \), and we can calculate its intersection number with \( \Lambda^{\geq 1}(\sigma) \), being
\[
\sum_{t \in (a',b')} \dim(\sigma_0(t) \cap \sigma) - \sum_{\lambda \in (\lambda_0, \lambda')} \dim(\sigma_{\lambda}(b) \cap \sigma) = 0.
\]

Since the sum indexed over \( (a',b') \) and \( (\lambda_0, \lambda') \) is the same as over \( (a,b) \) and \( \lambda < 0 \), we have that
\begin{equation} \label{negativeeigenvaluesformula}
\sum_{t \in (a,b)} \dim(\sigma_0(t) \cap \sigma) = \sum_{\lambda < 0} \dim(\sigma_{\lambda}(b) \cap \sigma).
\end{equation}

In particular, we now know that the number of negative eigenvalues with multiplicites of the Sturm-Liouville problem (\ref{sturmliouville}) is finite, given by (\ref{negativeeigenvaluesformula}). We finally prove this number is the index of \( H_b \):

\begin{prop}
\[
\ind(H_b) = \sum_{\lambda<0}\mathrm{nul}(H_b-\lambda I).
\]
\end{prop}

\begin{proof}
The solutions of (\ref{sturmliouville}) in \( \Gamma_0 \) for different \( \lambda \) are \( H_b \)-orthogonal, and the direct sum of the eigenspaces for the negative eigenvalues form a subspace on which \( H_b \) is negative definite. If \( \ind(H_b) \) were bigger than the number of negative eigenvalues, there would be a finite-dimensional subspace \( V \subseteq \Gamma_0 \) on which \( H_b \) is negative-definite and whose dimension is greater than this number. On it, the bilinear symmetric form
\[
\langle \langle X, Y \rangle \rangle \coloneqq \int_a^b \langle X, Y \rangle ds
\]
defines an inner product, so there exists a self-adjoint linear operator \( P: V \to V \) such that
\[
H_b(X,Y) = \langle \langle PX, Y \rangle \rangle = \int_a^b \langle PX, Y \rangle ds.
\]

Evidently \( P \) coincides with \( L_0[X] = -X'' + RX \) on \( V \), and by diagonalizing \( P \) in \( V \), we have a basis of orthonormal eigenvectors with corresponding negative eigenvalues. But this would imply that there are more negative eigenvalues counted with multiplicity than previously accounted for, a contradiction, so the equality in the proposition holds.
\end{proof}

With this last step, we have proved all the identities in theorem \ref{morseindex}, obtaining the desired result.

\end{document}